\documentclass[11pt]{amsart}
\usepackage{amsmath,amssymb,amsthm, verbatim, tikz, multicol, enumitem, subfigure}
\usepackage[foot]{amsaddr}
\usepackage[pagewise]{lineno}
\usetikzlibrary{through,calc,cd}
\usepackage{graphicx}
\usepackage{fancyhdr}

\newtheorem{theorem}{Theorem}         
\newtheorem{proposition}{Proposition}[section] 
\newtheorem{cor}[proposition]{Corollary}
\newtheorem{lemma}[proposition]{Lemma}

\numberwithin{equation}{section}

\theoremstyle{definition}
\newtheorem{example}[proposition]{Example}
\newtheorem{definition}[proposition]{Definition}
\newtheorem{remark}[proposition]{Remark}

\newcommand{\N}{\mathbb N}
\newcommand{\R}{\mathbb R}
\newcommand{\OO}{\mathcal O}
\newcommand{\I}{\mathcal I}
\newcommand{\Q}{\mathbb Q}
\newcommand{\Z}{\mathbb Z}

\DeclareMathOperator{\Enc}{Enc}
\DeclareMathOperator{\Dec}{Dec}
\DeclareMathOperator{\ran}{range}
\DeclareMathOperator{\PGL}{PGL}

\title{Slow continued fractions and permutative representations of $\OO_N$}

\author{Christopher Linden}
\address{ University of Illinois, Urbana Champaign}
\email{clinden2@illinois.edu}

\begin{document}
\begin{abstract}
Representations of the Cuntz algebra $\OO_N$ are constructed from interval dynamical systems associated with slow continued fraction algorithms introduced by Giovanni Panti. Their irreducible decomposition
formulas are characterized by using the modular group action on real
numbers, as a generalization of results by Kawamura, Hayashi and Lascu.
Furthermore, a certain symmetry of such an interval dynamical system is
interpreted as a covariant representation of the $C^*$--dynamical system of
the `flip-flop' automorphism of $\OO_2$.
\end{abstract}
\maketitle
\section{Introduction}
Permutative representations of the Cuntz algebras $\OO_N$ are a special class of representations arising from branching function systems. Bratteli and Jorgensen \cite{BraJorg} classified irreducible permutative representations of $\OO_{N}$ up to unitary equivalence. Kawamura, Hayashi, and Lascu \cite{KawaHayaLascu} studied the permutative representation arising from the Gauss map, a well-studied dynamical system related to continued fractions. They showed that unitary equivalence classes of irreducible permutative representations of $\OO_{\infty}$ correspond to $\PGL_2(\Z)$ equivalence classes of irrational numbers. Moreover, representations labeled by solutions to quadratic equations with integer coefficients are characterized by the existence of certain eigenvectors. This establishes a correspondence between the number theoretic properties of the label and the properties of the representation. If this connection is further developed, the rich combintorial and algebraic structure of continued fractions can be used to study the representation theory of Cuntz algebras.

In this note, we study permutative representations associated with so-called slow continued fraction algorithms (hereafter SCFAs) recently introduced by Panti \cite{Panti}. This is a broad class of continued fractions including regular continued fractions, Zagier's ceiling continued fractions \cite{Zag}, even and odd continued fractions \cite{BL}, and ``backwards" continued fractions \cite{AF}. Our main result is a correspondence between unitary equivalence classes of irreducible permutative representations of $\OO_{N}$ and $\PGL_2(\Z)$ equivalence classes of \emph{real} numbers for finite $N$ (Theorem \ref{mr}). In contrast to the $N= \infty$ case, it is insufficient to consider irrational numbers. The dynamics of rational numbers under the iteration of SCFAs is more complicated and less studied than that of irrational numbers, and is the source of several technical difficulties which we must overcome. 

 As the name suggests, SCFAs can be thought of as slow downs of analogues of the Gauss map. These ``Gauss maps'' can then be realized as jump transformations of SCFAs. We relate permutative representations associated to SCFAs and those associated to their jump transformations. Precisely, we show that the permutative representation of $\OO_N$ associated with an SCFA can be precomposed with an embedding of $\OO_{\infty}$ into $\OO_N$ to obtain the representation of $\OO_{\infty}$ associated with a jump transformation (Theorem \ref{thmjmp}). In a similar vein, we translate combinatorial relationships between different SCFAs into embeddings of $\OO_N$ into $\OO_2 \rtimes \Z_2$. We show that the representation of $\OO_N$ associated with any SCFA is the composition of this embedding and a representation of $\OO_2 \rtimes_{\theta} \Z_2$ (Theorem \ref{o2emb}). 
 
 The paper is organized in the following way: In section 2 we review Cuntz algebras and their permutative representations. In section 3 we discuss regular continued fractions and conclude with an outline of the argument of Kawamura, Hayashi, and Lascu \cite{KawaHayaLascu}. In section 4 we introduce the basic definitions and facts for SCFAs, largely following Panti \cite{Panti}. In section 5 we consider the symbolic dynamics of SCFAs, paying special attention to rational numbers. In section 6 we state and prove our main results connecting slow continued fractions and permutative representations of Cuntz algebras. In section 7 we consider a few examples of Theorem \ref{mr}.

\section{Permutative representations}
In this section we review Cuntz algebras and their permutative representations. For  $N= 2, 3, \dots, \infty$, the Cuntz algebra $\OO_N$ is the universal $C^*$-algebra generated by $\{S_i\}_{i=1}^N$ satisfying \cite{Cu} 
\begin{equation}\label{cuntz1}
S_i^*S_j = \delta_{ij} 1, \qquad \sum_{i=1}^N S_iS_i^* =1.
\end{equation} For $N = \infty$, the second equality is replaced with
$ \sum_{i=1}^n S_iS_i^* \leq 1$ for all $n \in \N$. Throughout this section we treat the finite and $N= \infty$ cases simultaneously. For convenience, we will write $\N_N$ for $\{1, 2, \dots N\}$, with the understanding that $\N_{\infty} = \N$.
\begin{definition}$\ $
\begin{enumerate}[label=({\roman*})] 
\item A \emph{representation} of a $C^*$-algebra $A$ on a Hilbert space $\mathcal{H}$ is a $*$-homomorphism from $A$ into $B(\mathcal{H})$, the set of all bounded linear operators on $\mathcal{H}$.
\item A subspace $V \subset \mathcal{H}$ is \emph{invariant} for a representation $\pi : A \to B(\mathcal{H})$ if $\pi(a)v \in V$ for any $a \in A$ and $v \in V$.
\item For a representation $\pi$ of $A$ on $\mathcal{H}$ with a closed invariant
subspace $V$, the restriction $\pi|_V$ of $\pi$ to $V$ is defined as the restriction of
the operator $\pi(a)$ to $V$ for each $a \in A$. We call  $\pi|_V : A \to B(V)$ a
\emph{subrepresentation} of $\pi$.
\item A representation $\pi : A \to B(\mathcal{H})$ is \emph{irreducible} if $\{0\}$ and $\mathcal{H}$ are the only closed invariant subspaces for $\pi$. 
\end{enumerate}
\end{definition}
 Any collection of isometries satisfying the relations (\ref{cuntz1}) determines a representation of $\OO_N$ because $\OO_N$ is simple \cite{Cu}. All $C^*$-algebras, representations, and embeddings which we consider are unital.
\begin{definition}\label{BFSdef} (\cite{BraJorg}, Chapter 2)
For $N= 2, 3, \dots, \infty$, a \emph{branching function system} (BFS) of order $N$ on a set $\Omega$ is a collection of injective transformations $\{f_i\}_{i=1}^N$ on $\Omega$ with pairwise disjoint ranges whose union is $\Omega$.
\end{definition}
We will refer to such a system by the tuple $\{\Omega, F, \{f_i\}_{i=1}^N, \{\Delta_i\}_{i=1}^N\}$ where $\Delta_i = f_i(\Omega)$, and $F$ is the piecewise function on $\Omega$  defined by $f_i^{-1}$ on $\Delta_i$.
\begin{definition}\label{Permdef} A \emph{permutative representation} of $\OO_N$ on $\mathcal{H}$ is a representation $\pi$ for which there is an orthonormal basis $\{e_k: k \in K \}$ for $\mathcal{H}$ such that 
\begin{equation}
\pi(S_i) e_n \in \{e_k : k \in K\} \qquad (n \in K, i \in \N_N).
\end{equation}
\end{definition}
\begin{proposition}[\cite{BraJorg}, p. 7]\label{BFS}
Let $\ell^2(\Omega)$ denote the Hilbert space with orthonormal basis $\{e_{\omega} : \omega \in \Omega\}$. Any BFS $\{\Omega, F, \{f_i\}_{i=1}^N, \{\Delta_i\}_{i=1}^N\}$ induces a  permutative representation $\pi_F: \OO_N \to B(\ell^2(\Omega))$ defined by
 \begin{equation}\label{BFSrepeqn}
 \pi_F(S_i)e_{\omega} = e_{f_i(\omega)} \qquad (\omega \in \Omega, i \in \N_{N}).
 \end{equation}
 \end{proposition}
 \begin{lemma}\label{eqlemma} If  $\{\Omega, F, \{f_i\}_{i=1}^N, \{\Delta_i\}_{i=1}^N\}$ and $\{\Omega', G, \{g_i\}_{i=1}^N, \{\Delta_i'\}_{i=1}^N\}$ are conjugate, then the representations $\pi_F$ and $\pi_G$ are unitarily equivalent.
 \end{lemma}
 \begin{proof}
 The function systems are conjugate if there exists a bijection  $C: \Omega \to \Omega'$ such that 
 \begin{equation}\label{conjlemma}
 C  \circ f_i = g_i \circ C \qquad i \in \N_N.
 \end{equation} Define the unitary $U: \ell^2(\Omega) \to \ell^2(\Omega')$ by $Ue_{\omega} = e_{C(\omega)}$. From (\ref{conjlemma}), we obtain \begin{equation}
 U\pi_F(S_i) = \pi_G(S_i)U \quad  (i \in \N_N).
 \end{equation}
 \end{proof}
\begin{definition} (\cite{BraJorg}, Chapter 4)
\begin{enumerate}[label=({\roman*})]
\item For a finite word $w = w_1w_2 \cdots w_k$ in the alphabet $\N_N$, we denote $S_{w_1}S_{w_2} \cdots S_{w_k}$ by $S_w$. Let $\mathcal{F}_N$ denote the $C^*$-subalgebra of $\OO_N$ generated by all elements of the form $S_wS_{w'}^*$, where $w$ and $w'$ are finite words with equal length in the alphabet $\N_N$.
\item  We call an irreducible permutative representation of $\OO_N$ a \emph{cycle}. 
\item We call an irreducible component of the restriction of a cycle to $\mathcal{F}_N$ an \emph{atom}.
\end{enumerate}
\end{definition} 
We recall a construction of the shift representation $\pi_S^N$ of $\OO_N$. Let $\Omega_N$ denote the set $N^{\N}$ of all infinite sequences in the alphabet $\{1, ..., N\}$. Define the BFS $\{\sigma_i\}_{i=1}^{N}$ on $\Omega_N$ by 
\begin{equation}
\sigma_i((x_1, x_2, x_3, ...)) =  (i,x_1,x_2, x_3, ...) \qquad (i \in \N_N).
\end{equation}
From Proposition \ref{BFS} we obtain a representation $\pi_S^N$ of $\OO_N$ on $\ell^2(\Omega_N)$ as
\begin{equation}
\pi_S^N(S_i)e_{(x_n)}= e_{\sigma_i((x_n))} \qquad (i \in \N_N \quad (x_n) \in \Omega_N).
\end{equation}

For $(x_n)$ and $(y_n)$ in $\Omega_N := N^{\N}$, write $(x_n) \sim (y_n)$ if there exist $z \in \Z, m \in \N$ such that  $x_{n+z} = y_n$ for $n \geq m$. Write $(x_n) \approx (y_n)$ if $z$ can be taken to be $0$. These are equivalence relations, which we call tail equivalence and eventual equivalence, respectively. We denote the equivalence class of $(x_n)$ under $\sim $ by $[(x_n)]$ and the equivalence class of $x$ under $\approx $ by $[[(x_n)]]$.
\begin{proposition}[\cite{BraJorg}, Chapter 6]\label{universality}

$\ $
\begin{enumerate}[label={\upshape(\roman*)}]

\item The decomposition of $\pi_S^N$ into cycles corresponds to the decomposition of $\ell^2(\Omega_N)$ into subspaces \begin{equation}\label{subspaces}
\ell^2(\Omega_N) = \bigoplus_{[(x_n)] \in \Omega_N/\sim} \mathcal{H}_{[(x_n)]}
\end{equation}
where $\mathcal{H}_{[(x_n)]}$ is the (separable) subspace of $\ell^2(\Omega_N)$ with basis $\{e_{(y_n)} : (y_n) \in [(x_n)]\}$.
\item Any irreducible permutative representation of $\OO_N$ is unitarily equivalent to exactly one such a representation.
\item The decomposition of $\pi_S^N$ (restricted to $\mathcal{F}_N$) into atoms corresponds to the decomposition of $\ell^2(\Omega_N)$ into subspaces \begin{equation}
\ell^2(\Omega_N) = \bigoplus_{[[(x_n)]] \in \Omega_N/\approx} \mathcal{H}_{[[(x_n)]]}
\end{equation}
where $\mathcal{H}_{[[(x_n)]]}$ is the subspace of $\ell^2(\Omega_N)$ with basis $\{e_{(y_n)} : (y_n) \in [[(x_n)]]\}$.
\end{enumerate}
\end{proposition}
\section{Regular Continued Fractions}
In this section we review a few useful definitions and facts about regular continued fractions. The regular continued fraction expansion of $x \in \I:= [0,1] \setminus \Q$ is written as
\begin{equation}\label{cfdef}
x=[a_1,a_2,\ldots]=\cfrac{1}{a_1+\frac{1}{\displaystyle{a_2+\frac{1}{\ddots}}} }
\end{equation}
where the partial quotients $a_i$ are positive integers. Continued fractions can be understood in terms of the Gauss and Farey maps, $G_R, F_R: \I \to \I$  defined as \begin{equation}
G_R(x)=\frac{1}{x} -\Bigl\lfloor{\frac{1}{x}}\Bigr\rfloor,
 \qquad
F_R(x)=\begin{cases}
\frac{x}{1-x} & \mbox{\rm if $x\in \left( 0,1/2 \right),$} \\
\frac{1-x}{x} & \mbox{\rm if $x\in \left( 1/2,1\right),$} \end{cases}
\end{equation}
where $\lfloor x \rfloor$ is the floor function. For $[a_1, a_2, \dots]$ as in (\ref{cfdef}) $G_R$ and $F_R$ act as:
\begin{eqnarray}\label{fareygaussdef}
G_R \big( [a_1,a_2,a_3,\ldots]\big) =[a_2,a_3,a_4,\ldots ],\\
\label{fareygaussdef2}
F_R \big( [a_1,a_2,a_3,\ldots ]\big) =\begin{cases}
[a_1 -1,a_2,a_3,\ldots ] & \mbox{\rm if $a_1 \geq 2$,} \\
[a_2,a_3,a_4,\ldots] & \mbox{\rm if $a_1=1$.} \end{cases}
\end{eqnarray}
Let $\N_0 := \{0\} \cup \N$. We have the relation
\begin{equation}\label{fareygauss}
G_R(x) = {F_R}^{r(x)+1}(x) \qquad \qquad r(x) =\inf\{ n \in \N_0 : {F_R}^n(x) \in (1/2, 1)\}.
\end{equation}
For $x$ in $\I$, (\ref{cfdef}) implies that $r(x) +1 = a_1$ and in particular $r(x)$ is finite. We say that the Gauss map is the \emph{jump transformation} of the Farey map obtained by inducing on the interval $(1/2, 1)$. From (\ref{fareygauss}), the Gauss map $G_R$ can be regarded as an acceleration of the Farey map $F_R$. Conversely, $F_R$ can be regarded as a slow-down of $G_R$.
\begin{definition}\label{pgl}
Let $\PGL_2(\Z)$ be the group of two by two integer matrices of determinant $\pm 1$, with the matrices $M$ and $-M$ identified. We will regard $\PGL_2(\Z)$ as the group of fractional linear transformations with integer coefficients by identifying the matrix   $ \left[ \begin{smallmatrix}
a & b \\
c & d \end{smallmatrix} \right]  \in \PGL_2(\Z)$ with the function $x \mapsto \frac{ax+b}{cx+d}$.
\end{definition}
\begin{definition}\label{groupequiv}
Let $\Sigma$ be a subgroup of $\PGL_2(\Z)$. We say $x, y \in \R$ are $\Sigma$-equivalent and write $x \sim_{\Sigma} y$ if there exists a matrix $ \left[ \begin{smallmatrix}
a & b \\
c & d \end{smallmatrix} \right] \in \Sigma$ such that $y = \frac{ax+b}{cx+d}$. This is an equivalence relation, and we denote the equivalence class of $x$ by $[x]_{\Sigma}$.
\end{definition}
We recall the following well-known facts about continued fractions:
\begin{proposition}\label{wellknown}(Theorems 2.3, 6.1, and 5.3 in Chapter 10 of \cite{Hua})
\begin{enumerate}[label={\upshape(\roman*)}]
\item Let $\Omega_{\infty}$ denote the set of all sequences of positive integers. The map $\I \ni x \mapsto (a_1, a_2, \dots) \in  \Omega_{\infty}$ defined by the correspondence in (\ref{cfdef}) is bijective.
\item (Lagrange) An irrational number has an eventually periodic continued fraction expansion if and only if it is the solution of a quadratic equation with integer coefficients.
\item (Serret, \cite{Serret}, p. 34) Two irrational numbers $x$ and $y$ have tail equivalent continued fraction expansions if and only if $x \sim_{\PGL_2(\Z)} y$.
\end{enumerate}
\end{proposition}
We end this section with a brief sketch of the argument of Kawamura, Hayashi, and Lascu in our terminology. The (regular) Gauss map $G_R$ is a branching function system on $\I = [0,1] \setminus \Q$, which the correspondence in Proposition \ref{wellknown}(i) conjugates to the full shift on $\Omega_{\infty} = \N ^{\N}$. By Lemma \ref{eqlemma}, the representation $\pi_{G_R}$ of $\OO_{\infty}$ associated to the Gauss map by Proposition \ref{BFS} is unitarily equivalent to the shift representation $\pi^{\infty}_S$. The conjugacy sends $\PGL_2(\Z)$-equivalence classes of irrational numbers to tail equivalence classes of sequences, as per Proposition \ref{wellknown}(iii). In light of Proposition \ref{universality}, we obtain a correspondence between unitary equivalence classes of irreducible permutative representations of $\OO_{\infty}$ and $\PGL_2(\Z)$-equivalence classes of irrational numbers. Proposition \ref{wellknown}(ii) implies that an irreducible permutative representation of $\OO_{\infty}$ has finitely many atoms if and only if it is labeled by a class of solutions to quadratic equations with integer coefficients.

\section{Slow Continued Fraction Algorithms}
In this section we introduce SCFAs, and discuss a few useful combinatorial properties.
\begin{definition}\label{unimodular}$\ $
\begin{enumerate}[label=({\roman*})]
\item A subinterval $[\frac{p}{q}, \frac{p'}{q'}] \subset [0,1]$ with rational endpoints is said to be \emph{unimodular} if $\frac{p}{q}$ and $\frac{p'}{q'}$  are reduced fractions such that $pq' - p'q = -1$.
\item  A \emph{unimodular partition} is a finite collection unimodular intervals $I_i$ whose union is $[0,1]$, such that for $i \neq j$, $I_i \cap I_j$ contains at most one point.
\end{enumerate}
\end{definition}
 \begin{definition}\label{scfa}
An SCFA is a finite collection of functions $h_i: [0,1] \to [0,1]$ such that 
\begin{enumerate}[label={\upshape(\roman*)}]
\item Each function $h_i$ is a fractional linear transformation in $\PGL_2(\Z)$.
\item The images $\{h_i([0,1])\}_{i=1}^N$ form a unimodular partition.
\end{enumerate}
\end{definition}
The most important example is the regular Farey map, with inverse branches $ \frac{x}{x+1}$ and $ \frac{1}{x+1}$. In general, the fractional linear transformation $x \mapsto \frac{ax+b}{cx+d}$ is continuous and monotone except at the singular point $-\frac{d}{c}$. The assumption $h_i([0,1]) \subset [0,1]$ ensures that this singularity does not occur on the interval $[0,1]$. Hence $h_i : [0,1] \to [0,1]$ is continuous and strictly monotone. By assumption, $h_i([0,1])$ is a unimodular interval, which we denote $[\frac{p_i}{q_i}, \frac{p_i'}{q_i'}	].$ If $h_i$ is increasing, then $h_i(\frac{0}{1}) = \frac{p_i}{q_i}$ and $h_i(\frac{1}{1}) = \frac{p_i'}{q_i'}$. We therefore have the formula 
\begin{equation}
h_i(x) =  \left[ \begin{array}{cc}
 p'_i-p_i & p_i \\
q'_i-q_i &q_i \end{array} \right](x).
\end{equation} If $h_i$ is decreasing, the situation is reversed and
\begin{equation}
h_i(x) =  \left[ \begin{array}{cc}
 p_i-p_i' & p_i' \\
q_i-q_i' &q_i' \end{array} \right](x).
\end{equation}
In general, denoting the determinant of $h_i$ (equivalently, the sign of its derivative), by $\epsilon_i \in \{\pm 1\}$, $h_i$ is given by the formula
\begin{equation}\label{matrixformula}
h_i(x) =  \left[ \begin{array}{cc}
 p'_i-p_i & p_i \\
q'_i-q_i &q_i \end{array} \right]\left[ \begin{array}{cc}
-1 & 1 \\
0 & 1 \end{array} \right]^{(1- \epsilon_i)/2}(x) .
\end{equation}
 Hence the data of a unimodular partition $\{[\frac{p_i}{q_i}, \frac{p_i'}{q_i'}	]\}_{i=1}^N$ and signs $\{\epsilon_i\}_{i=1}^N$, $\epsilon_i \in \{-1, 1\}$ specifies an SCFA. Our convention will be to order the unimodular partition (and hence the $h_i$) such that $0 = \frac{p_1}{q_1}$ and $\frac{p_i'}{q_i'} = \frac{p_{i+1}}{q_{i+1}}$ for $1 \leq i < N$. 
\begin{proposition}
Fix an SCFA $\{h_i\}_{i=1}^N$, and let $f_i$ be the restriction of $h_i$ to $\I:= [0,1] \setminus \Q$. The functions $\{f_i\}_{i=1}^N$ form a BFS on $\I$.
\end{proposition}
\begin{proof}
As we have already remarked, the condition  $h_i([0,1]) \subset [0,1]$ guarantees that $h_i$ is continuous and strictly monotone on $[0,1]$. Hence $f_i$ is injective. Since $[0,1] \subset \cup_{i=1}^N h_i([0,1])$ and $h_i$ maps irrational numbers to irrational numbers, $\I \subset  \cup_{i=1}^N f_i(\I)$. For $i \neq j$, $h_i([0,1]) \cap h_j([0,1])$ is either empty or a rational singleton. Hence $f_i(\I) \cap f_j(\I) = \emptyset$.
\end{proof}
\begin{definition}\label{jumps}  With the notation introduced in Definition \ref{BFSdef}(i), let $\{\I, F, \{f_i\}_{i=1}^N, \{\Delta_i\}_{i=1}^N\}$ be the BFS associated to an SCFA. Suppose $E \subset \I$ is of the form $\displaystyle{\cup_{i=j}^k \Delta_i}$ for some $1 \leq j \leq k \leq N$. Define
\begin{align}
{\I}_E &:= \{ x \in \I : F^n(x) \in E \text{ for infintely many } n \in \N, \\
r(x) &:= \inf\{ n \in \N_0 : F^n(x) \in E\} \text{ for } x \in {\I}_E.
\end{align}
The \emph{jump transformation} of $F$ induced on $E$ is $G(F,E): {\I}_E \to {\I}_E$  
\begin{equation}
G(F,E)(x) = F^{r(x)+1}(x).
\end{equation}
If $E$ is a proper subset of $\I$, then $G$ will have countable many inverse branches $g_j: {\I}_E \to {\I}_E$ which also form a BFS.
\end{definition}
\begin{figure}
\begin{multicols}{2}
\centering
\includegraphics[scale=0.4,bb = 0 0 360 220]{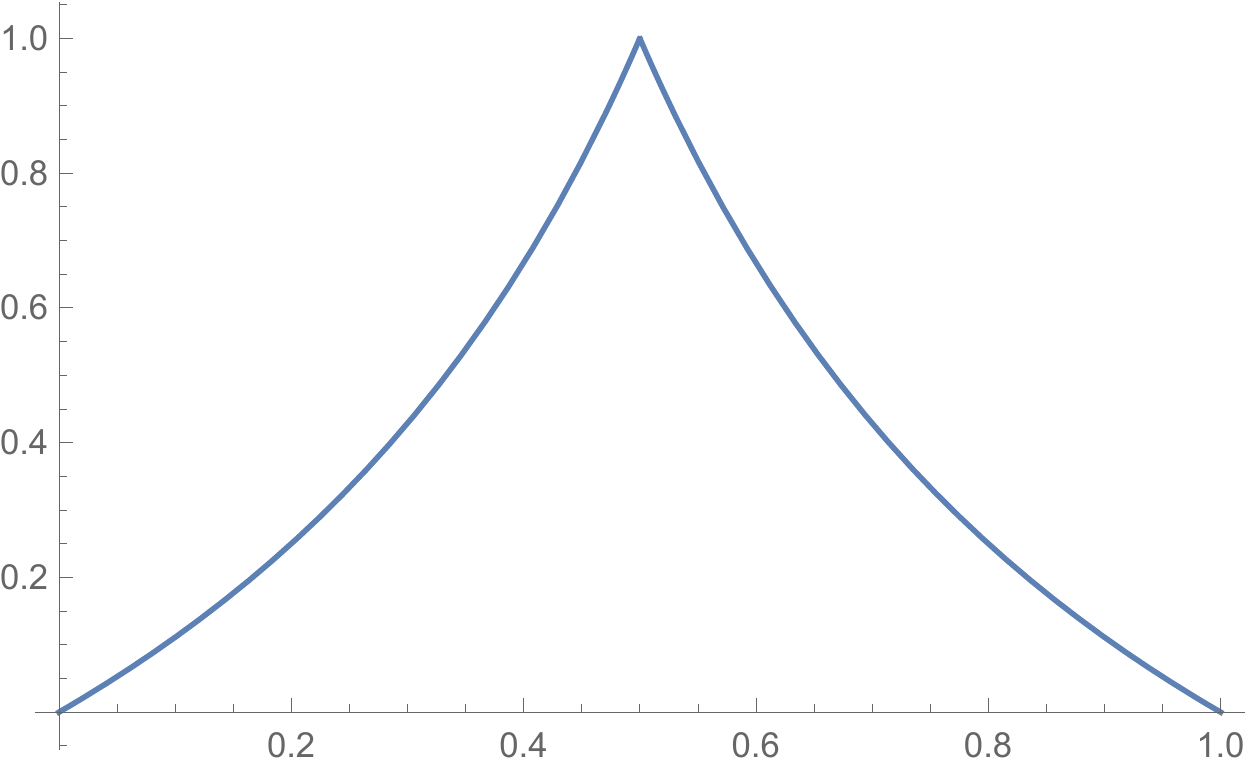}\\
$F_R$\\
\includegraphics[scale=0.4,bb = 0 0 360 220]{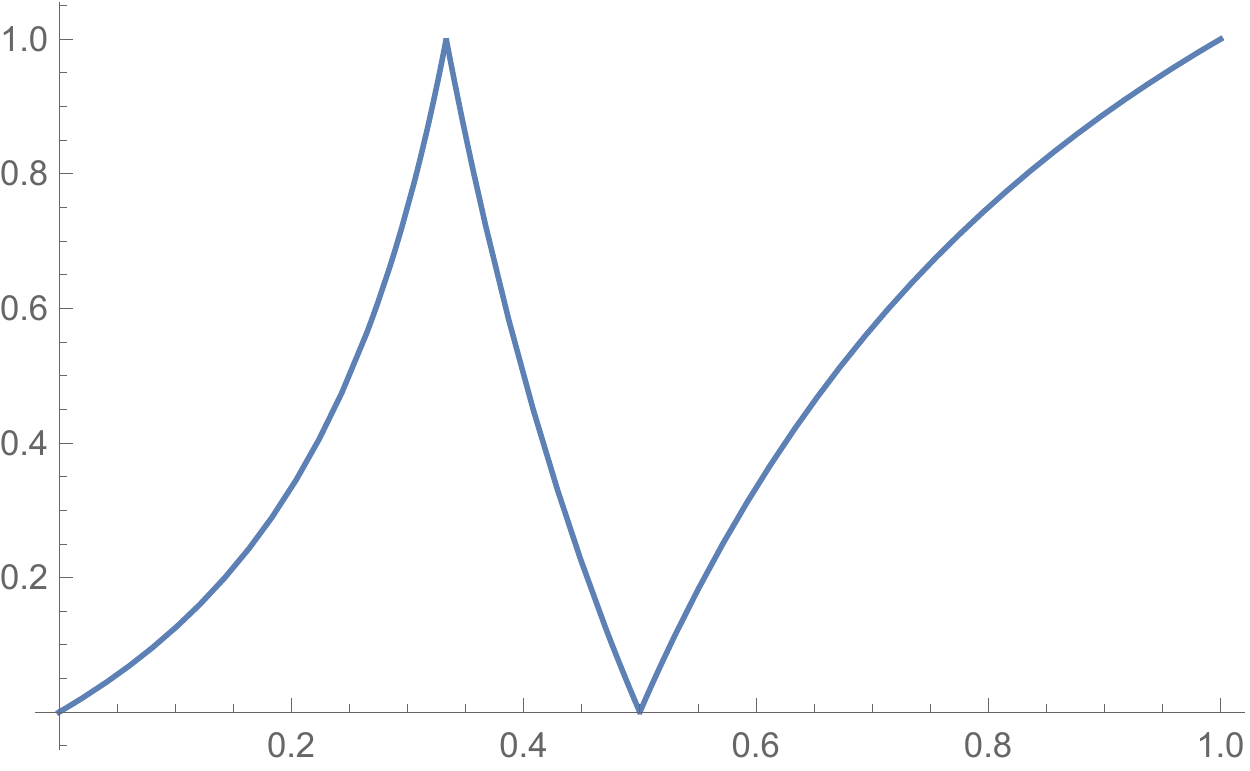}\\
$F_E$\\
\includegraphics[scale=0.4,bb = 0 0 360 220]{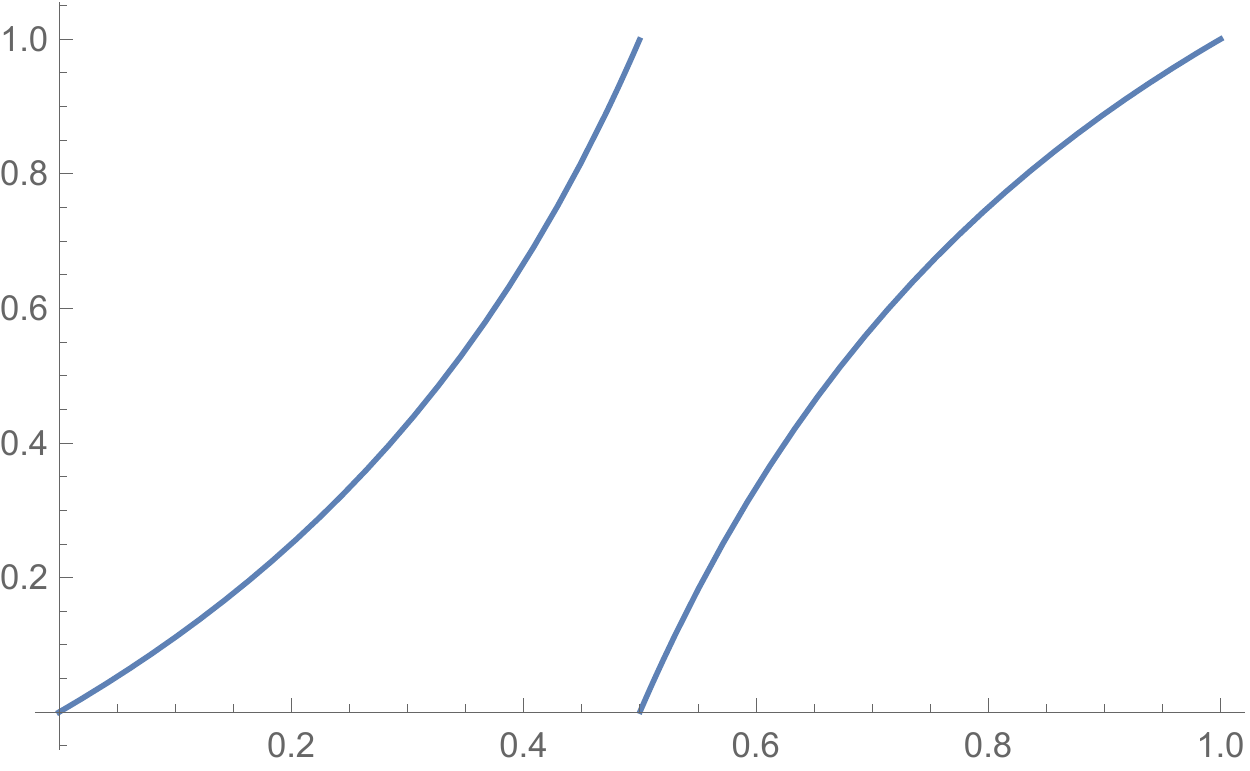}\\
$F_B$\\
\includegraphics[scale=0.4,bb = 0 0 360 220]{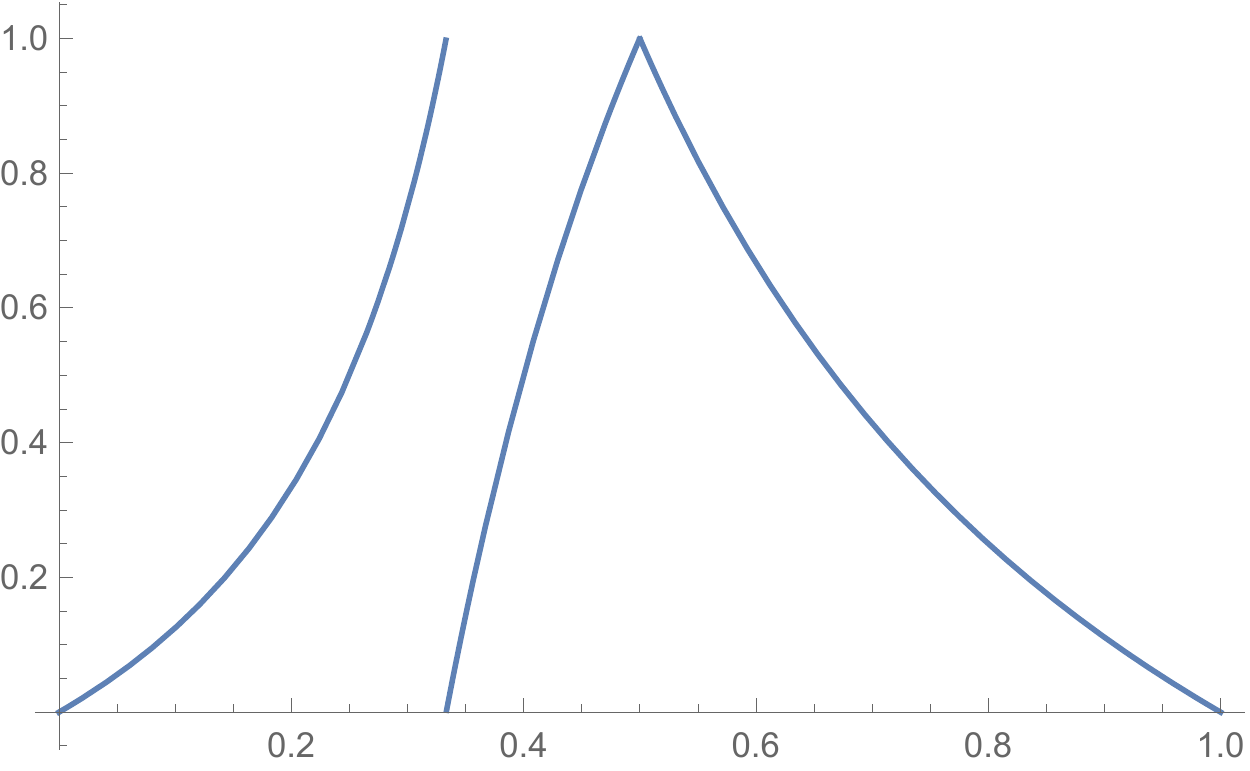}\\
$F_O$
\end{multicols}
\caption{The SCFAs of Example \ref{cfexamples}.}\label{figure1}
\end{figure}
With the above generalization of the relationship between the Farey and Gauss maps in hand, we can now describe several motivating examples of SCFAs and their jump transformations.
 \begin{example}\label{cfexamples}(See also Figure \ref{figure1}.)
 \begin{enumerate} [label=({\roman*})]
\item The classical Farey map $F_R$ in (\ref{fareygaussdef2}) is the SCFA associated with the partition $[0,1/2], [1/2,1]$ and signs $1,-1$. From (\ref{fareygauss}), inducing on $\Delta_2 =[1/2, 1] \cap \I$ yields the classical Gauss map $G_R$ as its acceleration \cite{AO, BY, H}.
\item An important SCFA which we denote $F_B$ is the SCFA associated with the partition $[0,1/2], [1/2,1]$ and signs $1,1$. Inducing $F_B$ on $\Delta_2 = [1/2, 1] \cap \I$ yields Zagier's ceiling algorithm $G(x) = \lceil \frac{1}{x} \rceil - \frac{1}{x}$, where $\lceil x \rceil$ is the ceiling function \cite{Zag}. Inducing on $\Delta_1 = [0, 1/2] \cap \I$ yields the ``backwards" continued fractions \cite{AF}.
\item The even and odd Farey maps $F_E$ and $F_O$  are  SCFAs associated with the partition $[0,1/3], [1/3,1/2], [1/2, 1]$ and signs $1,-1, 1$ and $1, 1, -1$, respectively. Inducing $F_E$ and $F_O$ on $\Delta_2 \cup \Delta_3= [1/3,1] \cap \I$ yields the even and odd Gauss maps \cite{BM, BL, Sch1}.
\end{enumerate}
\end{example}
The following lemma provides a useful description of the inverse branches of an arbitrary SCFA as compositions of those of $F_B$. Let $b_1 = \left[ \begin{smallmatrix}
1 & 0 \\
1 &1  \end{smallmatrix} \right]$ and $b_2 = \left[ \begin{smallmatrix}
0 & 1 \\
-1 &2 \end{smallmatrix} \right]$.
\begin{lemma}\label{unimodularstructure} Fix an SCFA $\{h_i\}_{i=1}^N$.
\begin{enumerate}[label=(\roman*)]
\item Each $ h_i =  \left[ \begin{smallmatrix}
 p'_i-p_i & p_i \\
q'_i-q_i &q_i \end{smallmatrix} \right]\left[ \begin{smallmatrix}
-1 & 1 \\
0 & 1 \end{smallmatrix}\right]^{(1-\epsilon_i)/2}$ can be written as $b_{\nu_i}T^{(1-\epsilon_i)/2}$ where $b_{\nu_i}$ is a word in $\{b_1, b_2\}$, $T(x) = 1-x$.
\item The words $\{b_{\nu_i}\}_{i=1}^N$ are the leaves of a finite, rooted binary tree. In particular, none of the words are left factors of another. The word $b_{\mu}b_1$ is a left factor of a word in $\{b_{\nu_i}\}_{i=1}^N$ if and only if $b_{\mu}b_2$ is a left factor of a word in $\{b_{\nu_i}\}_{i=1}^N$.
\end{enumerate}
\end{lemma}
\begin{proof}
 Any unimodular partition can be obtained uniquely from the interval $[\frac{0}{1},\frac{1}{1}]$ by repeatedly splitting an interval $[\frac{p}{q}, \frac{p'}{q'}]$ into two subintervals  $[\frac{p}{q}, \frac{p+p'}{q+q'}]$ and  $[\frac{p+p'}{q+q'}, \frac{p'}{q'}]$. Equation (\ref{matrixformula}) associates  the intervals  $[\frac{p}{q}, \frac{p'}{q'}]$, $[\frac{p+p'}{q+q'}, \frac{p'}{q'}]$, and  $[\frac{p}{q}, \frac{p+p'}{q+q'}]$ to the matrices $\left[ \begin{smallmatrix}
 p'-p & p \\
q'-q &q \end{smallmatrix} \right]$, $\left[ \begin{smallmatrix}
 -p & p +p' \\
 -q &  q + q'\end{smallmatrix} \right]$, and $\left[ \begin{smallmatrix}
 p' & p \\
q' &q \end{smallmatrix} \right]$.  The lemma follows from the observation that splitting an interval corresponds to right multiplication of its associated matrix with $b_1$ and $b_2$:
\begin{equation}
\left[ \begin{array}{cc}
 p'-p & p \\
q'-q &q \end{array} \right]
\left[ \begin{array}{cc}
1 & 0 \\
1 &1 \end{array} \right] = \left[ \begin{array}{cc}
 p' & p \\
q' &q \end{array} \right],
\end{equation}
\begin{equation}
\left[ \begin{array}{cc}
 p'-p & p \\
q'-q &q \end{array} \right]
\left[ \begin{array}{cc}
0 & 1 \\
-1 &2 \end{array} \right] = \left[ \begin{array}{cc}
 -p & p +p' \\
 -q &  q + q'\end{array} \right]. \qedhere
\end{equation}
\end{proof}
\section{Symbolic Dynamics for SCFAs}
 The analogue of continued fractions are \emph{itineraries} with respect to an SCFA. Although only irrational numbers have infinite continued fraction expansions, every real number will have an infinite itinerary. For the same reason that the (terminating) continued fraction expansions of rational numbers are not unique, each rational number in $(0,1)$ will have two itineraries. Instead of a bijection, we therefore work separately with a surjective decoding map and an injective encoding map.
\begin{definition}\label{itinerary}  Fix an SCFA $\{h_i\}_{i=1}^N$ with associated BFS $\{\I, F, \{f_i\}_{i=1}^N, \{\Delta_i\}_{i=1}^N\}$.
\begin{enumerate}[label=({\roman*})]
\item A sequence $(x_n) \in \Omega_N$ is an $F$-\emph{itinerary} for $x$ if $x \in \cap_{n=1}^{\infty} h_{x_n} \circ h_{x_{n-1}} \circ \cdots \circ h_{x_1}([0,1])$.
\item We write $x \sim_F y$ if there exist tail equivalent $F$-itineraries for $x$ and $y$. This is an equivalence relation, and we write $[x]_F$ for the $\sim_F$-equivalence class of $x$.
\item The Panti-Serret group $\Sigma_F$ associated with $\{\I, F, \{f_i\}_{i=1}^N, \{\Delta_i\}_{i=1}^N\}$ is the subgroup of $\PGL_2(\Z)$ generated by the matrices $h_i$.
\item The SCFA $\{\I, F, \{f_i\}_{i=1}^N, \{\Delta_i\}_{i=1}^N\}$ is said to satisfy the Serret theorem if for irrational numbers, the relation $\sim_{F}$ coincides with $\sim_{\Sigma_F}$ as in Definition \ref{groupequiv}.
\end{enumerate}
\end{definition}
The Serret theorem in Definition \ref{itinerary}(iv) holds for some, but not all SCFAs. In general, each $\Sigma_F$ equivalence class is the union of $\sim_F$ equivalence classes. A practical criterion for checking the validity of the Serret theorem is established in \cite{Panti}.
We recall the following facts from \cite{Panti2}:
\begin{proposition}\label{lagrange}(Panti)
\begin{enumerate}[label={\upshape(\roman*)}]
\item (Observation 3 in \cite{Panti2}) The intersection in Definition \ref{itinerary}(i) is always a singleton.
\item  (Generalized Lagrange Theorem, Section 3 of \cite{Panti2}) An irrational number has an eventually periodic itinerary with respect to every SCFA if and only if it is a solution to a quadratic equation with integer coefficients.
\end{enumerate}
\end{proposition}
\begin{definition}\label{encodedecode}
Fix an SCFA $\{h_i\}_{i=1}^N$ with BFS $\{\I, F, \{f_i\}_{i=1}^N, \{\Delta_i\}_{i=1}^N\}$. 
\begin{enumerate}[label=({\roman*})]
\item For $(x_n) \in \Omega_N$, let $x$ be the singleton $\cap_{n=1}^{\infty} h_{x_n} \circ h_{x_{n-1}} \circ \cdots \circ h_{x_1}([0,1])$, as in Proposition \ref{lagrange}(i). Define the \emph{decoding map} $\Dec_F: \Omega_N \to [0,1]$ by $\Dec_F((x_n)) =x$.
\item For $x \in \I$, let $(x_n) \in \Omega_N$ be the sequence such that $F^{n-1}(x) \in \Delta_{x_n}$. Define the \emph{encoding map} $\Enc_F: \I \to \Omega_N$ by $\Enc_F(x) = (x_n)$.
\end{enumerate}
\end{definition}
\begin{remark}\label{decodesurj}
Since $\cup_{i=1}^N h_i([0,1]) = [0,1]$, $\Dec_F$ is surjective. The injectivity of $\Enc_F$ is Consequence (i) of Observation 3 in \cite{Panti2}.
\end{remark}
\begin{proposition}\label{subshift}
The encoding map $\Enc_F$ conjugates $\{\I, F, \{f_i\}_{i=1}^N, \{\Delta_i\}_{i=1}^N\}$ to a subshift $\{\Enc_F(\I), \sigma, \{\sigma_i\}_{i=1}^N, \{\sigma_i(\Enc_F(\I))\}_{i=1}^N\}$.
\end{proposition}
\begin{proof} Since $\Enc_F$ is injective, it is enough to observe that 
	\begin{equation}\label{critobs}
	\Enc_F \circ f_i = \sigma_i \circ \Enc_F \qquad  i \in \{ 1, \dots, N\}.
\end{equation}\end{proof}
The encoding map in the above proposition is never surjective; its image is always $\Omega_N$ minus the  itineraries of rational numbers. For example, denoting the constant sequences by $\overline{1} = 1,1,1, \dots$ and $\overline{2} = 2,2,2, \dots$,
		\begin{eqnarray}
		\Enc_{F_R}(\I) = \Omega_2 \setminus \{(x_n) :(x_n) \sim \overline{1}\},\\
		\Enc_{F_B}(\I) = \Omega_2 \setminus \{(x_n) : (x_n) \sim \overline{2} \text{ or } (x_n) \sim \overline{1}\}.
		\end{eqnarray}	
	
We conclude this section by establishing a strong version of the Serret theorem which includes rational numbers for a specific family of SCFAs which we denote $F_N$.
\begin{definition}\label{fndef}
For $2 \leq N < \infty$ let $F_N$ to be the SCFA associated with the unimodular partition $[0, \frac{1}{N}], [\frac{1}{N}, \frac{1}{N-1}], [\frac{1}{N-1}, \frac{1}{N-2}], \dots, [\frac{1}{2}, \frac{1}{1}]$ and  signs $1, -1, -1, \dots, -1$.
\end{definition}
\begin{proposition}\label{fnprop}$\ $
\begin{enumerate}[label={\upshape(\roman*)}]
\item The subgroup $\Sigma_{F_N}$ of $\PGL_2(\Z)$ generated by  $\{h_i\}_{i=1}^N$ coincides with $\PGL_2(\Z)$.
\item The Serret theorem holds for $F_N$, in the sense of Definition \ref{itinerary}(iv).
\item A sequence $(x_n) \in \Omega_N$ is an $F_N$-itinerary of a rational number if and only if $(x_n) \sim \overline{1} = 1,1,1, \dots$.
\end{enumerate}
\end{proposition}
\begin{cor}\label{strongserret} Let $x, y \in [0,1]$ and $(x_n), (y_n) \in \Omega_N$ such that $\Dec_{F_N}((x_n)) =x$ and $\Dec_{F_N}((y_n)) =y$. Then $(x_n) \sim (y_n)$ if and only if $x \sim_{\PGL_2(\Z)} y$ . 
\end{cor}
\begin{proof}[Proof of Proposition \ref{fnprop}]$\ $
\begin{enumerate}[label={\upshape(\roman*)}]
\item Note that $F_2$ is simply the Farey map $F_R$. Denote the inverse branches of $F_R$ by $r_1$ and $r_2$. The inverse branches of $F_N$ are $r_1^{N-1}$, $r_1^{N-2}r_2$, $r_1^{N-3}r_2$, $\dots$, $r_1^2r_2$,  $r_1r_2$, $r_2$. Proposition \ref{fnprop}(i) then follows from the fact that $\{ r_1, r_2\}$ generates $\PGL_2(\Z)$.
\item We construct the transducer (finite state automaton) considered in Lemma 5.5 of \cite{Panti}. Let $b_1$ and $b_2$ be the branches of $F_B$ and $T(x) = 1-x$, as in Lemma \ref{unimodularstructure}. Given $h \in \{h_i\}_{i =1}^{N}$ and $v \in V$, there is a unique $w \in V$ for which there exists a (possibly empty) word $\mu = \mu_1\mu_2\cdots \mu_{n}$ in the alphabet $\{1, ... N\}$ such that \begin{equation}\label{edge}
vh = h_{\mu_1} h_{\mu_2} \cdots h_{\mu_n} w.
\end{equation}

The transducer in question has state set $V =\{b_1^kT^{e} : 0\leq k \leq N-2,e \in \{0, 1\}\}$. For $v$ and $w$ as in \ref{edge}, it has a directed edge from $v$ to $w$ labeled with input $h$ and output $h_{\mu_1} h_{\mu_2} \cdots h_{\mu_n}$. To construct the edge set, we consider cases:
\begin{enumerate}[label=({\alph*})]
\item If $v = b_1^k$ and $h = b_1^{N-1}$: We obtain a self loop, labeled with both input and output $h$.
\item If $v = b_1^k$ and $h = b_1^{j}b_2T$: We obtain an edge to $1$. If $k=0$ (i.e., $v=1$) this self loop is labeled with both input and output $h$.
\item If $v = b_1^kT$ and $h = b_1^{N-1}$ : We obtain an edge to $b_1^{N-2}$.
\item If $v = b_1^kT$ and  $h = b_1^{j}b_2T$ for $j \neq 0 $: We obtain an edge to $1$.
\item  If $v = b_1^kT$ and $h =b_2T$, for $ k \neq  N-2$: We obtain an edge to $b_1^{k+1}$.
\item  If $v = b_1^{N-2}T$ and  $h =b_2T$: We obtain an edge to $1$.
\end{enumerate}

An infinite path in the transducer constructed above eventually consists of an infinitely repeated self loop, labeled with the same input as output. Hence the output of the transducer is always tail equivalent to its input. By Corollary 5.6 of \cite{Panti}, the Serret theorem holds.
\item This is a special case of the following Lemma \ref{rationallemma}.
\end{enumerate}		
\end{proof}	
\begin{lemma}\label{rationallemma}
Suppose the SCFA $\{h_i\}_{i=1}^N$ satisfies $\epsilon_1 =1$ and $\epsilon_N = -1$. (Recall our ordering convention that $0 \in h_1([0,1])$ and $1 \in h_N([0,1])$.)  A sequence $(x_n) \in \Omega_N$ is an itinerary of a rational number if and only if $(x_n) \sim \overline{1} = 1,1,1, \dots$.
\end{lemma}
\begin{proof}
By the proof of Observation $3$ in \cite{Panti2}, for $r \in \Q$ and $(x_n) \in \Omega_N$, there is $M \in \N$ such that $r$ is not in the topological interior of $h_{x_m} \circ \cdots \circ h_{x_1}([0,1])$ for $m \geq M$. If $(x_n)$ is an itinerary for $x \in \Q \cap [0,1]$, this implies that $x$ is an endpoint of $h_{x_m} \circ \cdots \circ h_{x_1}([0,1])$ for $m \geq M$. If the determinant of $h_{x_m} \circ \cdots \circ h_{x_1}$ is $1$, then $h_{x_m}  \circ \cdots \circ h_{x_1}([0,1])$ shares its right endpoint with $h_{N} \circ h_{x_{m}} \circ \cdots \circ h_{x_1}([0,1])$ and left endpoint with $h_{1} \circ h_{x_{m}} \circ \cdots \circ h_{x_1}([0,1])$. If the determinant is $-1$, the situation is reversed. In this way, the shared endpoint and the determinant of $h_{x_m} \circ \cdots \circ h_{x_1}$ inductively determine $x_{m+1}$.

If $x_m = N$, the determinants of $h_{x_m} \circ \cdots \circ h_{x_1}$ and $h_{x_{m-1}} \circ \cdots \circ h_{x_1}$ differ by the assumption $\det(h_N) =-1$. By the above, $x_{m+1} = 1$.  If $x_m =1$, the determinants of $h_{x_m} \circ \cdots \circ h_{x_1}$ and $h_{x_{m-1}} \circ \cdots \circ h_{x_1}$ coincide by the assumption $\det(h_1) =1$. Again, $x_{m+1} =1$. We conclude that $x_m = 1$ for $m > M$, proving the forward implication.

 Conversely, if $(x_n)$ is tail equivalent to $\overline{1}$, then for $n$ and $m$ sufficiently large, the intervals $h_{x_n} \circ \cdots \circ h_{x_1}([0,1])$ and $h_{x_m} \circ \cdots \circ h_{x_1}([0,1])$ share a rational endpoint. The common endpoint is therefore the unique point in the intersection $\cap_{n=1}^{\infty}  h_{x_n} \circ  \cdots \circ h_{x_1}([0,1])$. We conclude that $(x_n)$ is the itinerary of a rational number, proving the lemma.
\end{proof}
\section{Main Results}\label{section4}
In this section we state and prove our main results. We begin with the following corollary of Lemma \ref{eqlemma} and Propositions \ref{universality} and \ref{subshift}.
\begin{proposition}\label{rep} Let $\{\I, F, \{f_i\}_{i=1}^N, \{\Delta_i\}_{i=1}^N\}$ be the BFS associated to an SCFA and $\pi_F$ the permutative representation of $\OO_N$  on $\ell^2(\I)$ in Proposition \ref{BFS}. Then the irreducible decomposition of $\pi_F$ is given as
\begin{equation}\label{fdecomp}
\ell^2(\I) = \bigoplus_{[x]_F \in \I / \sim_F} \mathcal{H}_{[x]_F}
\end{equation}
where $\mathcal{H}_{[x]_F}$ is the subspace of $\ell^2(\I)$ with basis $\{e_y : y \sim_F x \}$.
	\end{proposition} 
	\begin{remark}
	If the SCFA satisfies the Serret theorem as in Definition \ref{itinerary}(iv), then the sets $[x]_F$ coincide with $\Sigma_F$-orbits of irrational numbers.
	\end{remark}
	We now consider the SCFAs $F_N$ introduced in Definition \ref{fndef} to produce a bijection between equivalence classes of irreducible permutative representations of $\OO_N$ and $\PGL_2(\Z)$-equivalence classes of real numbers. 
	\begin{theorem}\label{mr}
For $2 \leq N < \infty$, the decoding map $\Dec_{F_N}$ in Definition \ref{encodedecode}(i) provides a bijection between unitary equivalence classes of irreducible permutative representations of $\OO_N$ and $\PGL_2(\Z)$-equivalence classes of real numbers. Moreover, an equivalence class of representations corresponds to an equivalence class of solutions to quadratic equations with integer coefficients if and only if it has finitely many atoms.
\end{theorem}
\begin{proof}
By Proposition \ref{universality}, there is a bijection between unitary equivalence classes of irreducible permutative representations of $\OO_N$ and the subspaces $\mathcal{H}_{[(x_n)]}$ of $\ell^2(\Omega_N)$. We consider the bijection 
\begin{equation}
\mathcal{H}_{[(x_n)]} \mapsto [\Dec_{F_N}((x_n))]_{\PGL_2(\Z)} \in [0,1] / \sim_{\PGL_2(\Z)}.
\end{equation}
This is well defined by the forward implication in Corollary \ref{strongserret}. It is injective by the backwards implication in Corollary \ref{strongserret}, and surjective by Remark \ref{decodesurj}. By Proposition \ref{lagrange}(ii) for irrationals and Proposition \ref{fnprop}(iii) for rationals, $[(x_n)]$ is labeled by a class of solutions to quadratic equations with integer coefficients if and only if $[(x_n)]$ contains a periodic sequence. This is equivalent to consisting of finitely many eventual equivalence classes, which by Proposition \ref{universality}(iii) is equivalent to the associated irreducible permutative representation of $\OO_N$ having finitely many atoms. \end{proof}
\begin{remark}\label{remarkcoresp}
If $x \in [0,1]$ has $F_N$-itinerary $(x_n)$ and $w_n$ is the word $x_1x_2 \cdots x_n$, then irreducible permutative representations of $\OO_N$ labeled by $[x]_{\PGL_2(\Z)}$ are characterized by the existence of a vector $\xi$ such that $S_{w_n}^* \xi \neq 0$ for $n \in \N$. For $\pi_S^N|_{\mathcal{H}_{[(x_n)]}}$, $\xi$ is simply $e_{(x_n)}$. We consider examples in Section \ref{mrex}.
\end{remark}
	\begin{theorem}\label{thmjmp}
	Let  $\{\I, F, \{f_i\}_{i=1}^N, \{\Delta_i\}_{i=1}^N\}$ be the BFS associated to an SCFA, and $G= G(F,E)$ be the jump transformation of $F$ induced on $E$. There is a unital embedding $\varphi:\OO_{\infty} \to \OO_N$ which is compatible with the representations $\pi_F$ and $\pi_G$ (as defined in Proposition \ref{BFS}) in the sense that $\ell^2({\I}_E)$ sits inside $\ell^2(\I)$ as a closed, invariant subspace on which 
	\begin{equation}\label{comp}
	\pi_F \circ \varphi_F = \pi_G.
	\end{equation}
	\end{theorem}
\begin{proof}
From Definition \ref{jumps}, ${\I}_E = \{x \in \I : F^n(x) \in E \text{ for infinitely many } n \in \N \}$, and hence $\ell^2(\I_{E})$ is invariant under the action of $\pi_F(\OO_N)$.  Let $\{T_j : j \in \N\}$ be the generators of $\OO_{\infty}$ and $\{S_i: i = 1, \dots, N\}$ be the generators of $\OO_N$. Let $f$  $ = \{ f_i : \ran(f_i) \subset E\}$ and $f_{E^c} = \{ f_i : \ran(f_i) \cap E = \emptyset\}$. The inverse branches $g_j$ of $G(F,E)$ are of the form $g_j =f_{j_1} \circ f_{j_2} \circ \cdots \circ f_{j_k}$ where $f_{j_k} \in f_E$ and $f_{j_1}, \dots, f_{{j_{k-1}}} \in f_{E^c}$. We denote the word $j_1 j_2 \cdots  j_k$ by $\mu_j$. Define $\varphi(1) =1$ and
\begin{equation}
\varphi(T_j) = S_{\mu_j},
\end{equation}
from which (\ref{comp}) immediately follows. To show that $\varphi$ is an embedding, it suffices to verify (\ref{cuntz1}). For $j, j' \in \N$, only $f_{j_k}$ and $f_{{j'}_{k'}}$ belong to $f_E$. Hence $\mu_j$ is a left factor of $\mu_{j'}$ only if $j=j'$. This verifies the left hand equality of (\ref{cuntz1}):
\begin{equation}
\varphi(T_{j'}^*)\varphi(T_j) = S_{\mu_{j'}}^*S_{\mu_j} = \delta_{j, j'}1.
\end{equation}
For $j \neq j'$, $\mu_j$ is not a left factor of $\mu_{j'}$ nor vice versa so $\varphi(T_j)\varphi(T_j)^*$ and $\varphi(T_{j'})\varphi(T_{j'})^*$ are projections with disjoint ranges, verifying the right hand inequality of (\ref{cuntz1}).
\end{proof}
\begin{remark}
A theorem of a similar flavor appears in \cite{KawaLascuColt}, relating representations associated with the regular Gauss and Farey maps.
\end{remark}
The `flip-flop' \cite{Arch} automorphism $\theta$ of $\OO_2$ is defined by $\theta(S_1) =S_2$ and $\theta(S_2) =S_1$. Since $\theta$ is an involution, it determines an action of the group $\Z_2 := \Z / 2\Z$ on $\OO_2$. We write $\OO_2 \rtimes_{\theta} \Z_2$ for  the associated crossed product. We refer to Chapter 2, Section 2.3 of \cite{Williams} for a treatment of crossed products of $C^*$-algebras by finite groups. Recall the SCFA $F_B$ introduced in Example \ref{cfexamples}(ii).
	\begin{proposition} The representation $\pi_{F_B}$ of $\OO_2$ on $\ell^2(\I)$ extends to a representation $\tilde{\pi}_{F_B}$ of $\OO_2 \rtimes_{\theta} \Z_2$ on $\ell^2(\I)$.
\end{proposition}
\begin{proof} 
 Let $U_{\theta} \in \OO_2 \rtimes_{\theta} \Z_2$ be the unitary which implements $\theta$, i.e.
 \begin{equation}
 U_{\theta}A U_{\theta}^* = \theta(A), \qquad A \in \OO_2.
 \end{equation} 
 Since $S_1$ and $S_2$ generate $\OO_2$ and $U_{\theta} = U_{\theta}^*$, this is equivalent to
 \begin{equation}
  U_{\theta}S_1 = S_2 U_{\theta}.
 \end{equation}
   Elements of $\OO_2 \rtimes_{\theta} \Z_2$ can be written in the form $A + U_{\theta}B$, where $A,B \in \OO_2$. Therefore any self-adjoint unitary $U:\ell^2(\I) \to \ell^2(\I)$ which satisfies  
 \begin{equation}
 U \pi_{F_B}(S_1) = \pi_{F_B}(S_2)  U
 \end{equation}
defines an extension $\tilde{\pi}_{F_B}$ by setting $\tilde{\pi}_{F_B}(U_{\theta}) =U$. Define the self-adjoint unitary $Ue_x = e_{1-x}$. Equations (\ref{BFSrepeqn}) and (\ref{matrixformula}) yield
\begin{equation}
\pi_{F_B}(S_1)e_x = e_{\frac{x}{1+x}}, \qquad
\pi_{F_B}(S_2)e_x = e_{\frac{1}{2-x}}.
\end{equation}
Applying these, 
\begin{equation}
U\pi_{F_B}(S_1) = \pi_{F_B}(S_2) U,
\end{equation}
and hence $\tilde{\pi}_{F_B}(U_{\theta})e_x = e_{1-x}$ gives the claimed extension.
\end{proof}
	\begin{theorem}\label{o2emb} Let $\{\I, F, \{f_i\}_{i=1}^N, \{\Delta_i\}_{i=1}^N\}$ be the BFS associated to an SCFA. There is a unital embedding $ \psi_F :\OO_N \to \OO_2 \rtimes_{\theta} \Z_2$ such that the following diagram commutes:\\
	\begin{center}
	\begin{tikzcd}
\OO_2 \rtimes_{\theta} \Z_2  \arrow[r, "\tilde{\pi}_{F_B}"] & B(\ell^2(\I)) \\
 \OO_N \arrow[u, "\psi_F"] \arrow[ur,"\pi_{F}"]&
\end{tikzcd}
\end{center}
	\end{theorem}
\begin{proof}
Let $\{S_i\}_{i=1}^N$ be the generators of $\OO_N$, and $B_1, B_2$, and $U_{\theta}$ the generators of $ \OO_2 \rtimes_{\theta} \Z_2$. Applying Lemma \ref{unimodularstructure}(i), write $f_i = b_{\nu_i} T^{e_i}$ and define
	\begin{equation}
	\psi_F(S_i) = B_{\nu_i} U_{\theta}^{e_i}
	\end{equation}
which immediately satisfies $\tilde{\pi}_{F_B} \circ \psi_F = \pi_F $. Consider 
\begin{equation}
\psi(S_{i'}^*)\psi(S_i) = U_{\theta}^{e_{i'}} B_{\mu_{i'}}^*   B_{\nu_i} U_{\theta}^{e_i}.
\end{equation} This expression is nonzero only if $\mu_{i}$ is a left factor of $\mu_{i'}$ or vice versa. By Lemma \ref{unimodularstructure}(ii) this occurs only when $i = i'$. Hence 
\begin{equation}
\psi(S_{i'}^*)\psi(S_i) = \delta_{i,i'}1,
\end{equation}
verifying the left-hand side of (\ref{cuntz1}). Now consider \begin{equation}\sum_{1 \leq i \leq N} \psi_F(S_{i})\psi_F(S_i)^* = \sum_{1 \leq i \leq N} B_{\nu_i} U_{\theta}^{p_i}U_{\theta}^{p_i}
B_{\nu_i}^* =  \sum_{1 \leq i \leq N} B_{\nu_i} B_{\nu_i}^*.
\end{equation}
Repeated application of $B_1B_1^* + B_2B_2^* =1$, together with the binary tree structure described in Lemma \ref{unimodularstructure}(ii), shows that \begin{equation} \sum_{1 \leq i \leq N} B_{\nu_i} B_{\nu_i}^* =1,
\end{equation}
verifying the right-hand side of (\ref{cuntz1}).
	\end{proof}
\begin{remark} If the signs $\epsilon_i$ of the SCFA are all positive, Theorem \ref{o2emb} gives an embedding into $\OO_2$.
\end{remark}

\section{Examples}\label{mrex}
In this section we give examples of Theorem \ref{mr}. We choose a representative of a $\PGL_2(\Z)$-equivalence class from $[0,1]$ and characterize the associated unitary equivalence class of irreducible permutative representations of $\OO_N$ for $N =2,3, \dots$. 
\begin{example}
For $N =2,3, \dots$, the ${F}_N$-itinerary of $0$ is $\overline{1} = 1,1,1,\dots$. The irreducible permutative representations of $\OO_N$ labeled by $0$ are characterized by the existence of a vector $\xi$ such that $S_1\xi = \xi$, and consist of a single atom.
\end{example}
\begin{example}
Similarly, for $N =2,3, \dots$, the ${F}_N$-itinerary of $\frac{\sqrt{5} -1}{2}$ is $\overline{N}$, and the irreducible permutative representations of $\OO_N$ labeled by $\frac{\sqrt{5} -1}{2}$ are characterized by the existence of a vector $\xi$ such that $S_N \xi = \xi$, and consist of a single atom.
\end{example}
\begin{example}
For $N=2$, $\sqrt{2} -1$ has ${F}_2$-itinerary $\overline{12}$, so the corresponding irreducible permutative representations of $\OO_2$ are characterized by the existence of a vector $\xi$ such that $S_2S_1 \xi = \xi$. There are two atoms, corresponding to the two eventual equivalence classes. For $2 < N < \infty$, $\sqrt{2}-1$ has $F_N$-itinerary $\overline{N-1}$, so the corresponding irreducible permutative representations of $\OO_N$ are characterized by the existence of a vector $\xi$ such that $S_{N-1} \xi = \xi$, and consist of a single atom.
\end{example}
\begin{remark}
The existence of an eigenvector for a finite composition of the generating isometries is how Hayashi, Kawamura, and Lascu characterize representations associated with quadratic irrationals for $\OO_{\infty}$. This is equivalent to having finitely many atoms, which is the characterization established in Theorem \ref{mr} for finite $N$.
\end{remark}
Finally, we consider a label which is not a quadratic root, for which the corresponding irreducible permutative representations must have countably many atoms.
\begin{example}
Let $e$ be the base of the natural logarithm. The regular continued fraction expansion of $e-2$ is $[1,2,1, 1,4,1, 1,6,1,1,8,1, \dots]$. The ${F}_2$-itinerary of $e-2$ is $$2,1,2,2,2,1,1,1,2,2,2,1,1,1,1,1,2, 2,1,1,1,1,1,1,2,2,\dots.$$ For each $N$ the itinerary is aperiodic. The ${F}_3$, ${F}_4$, and  ${F}_5$-itineraries of $e-2$ are respectively
\begin{eqnarray*}
& 3,2,3,3,1,1,2,3,3,1,1,1,1,2,3, 3,1,1,1,1,1,1,2,3, \dots,\\
& 4,3,4, 4, 1,2,4, 4,1,1,1,2,4, 1,1,1,1,1,2,4, \dots,\\
& 5,4,5,5,2,5,5,1,1,2,5,1,1,1,1,2,5,\dots.
\end{eqnarray*}
Every finite composition of the generating isometries of the associated representation of $\OO_N$ is a pure isometry. In particular, there are no eigenvectors as in the previous examples. The $F_N$-itineraries nevertheless characterize the representations as discussed in Remark \ref{remarkcoresp}.
\end{example}


\begin{thebibliography}{99}
%

\bibitem{AF}
R. Adler, L. Flatto,  \emph{Geodesic flows, interval maps, and symbolic dynamics}, Bull. Amer. Math. Soc. (N.S.) \textbf{25} (1991), no. 2, 229--334. 

\bibitem{AO}
H. Appelgate, H. Onishi, \emph{The Slow Continued Fraction Algorithm Via 2 $\times$ 2 Matrices}, Amer. Math. Monthly \textbf{90} (1983), no. 7, 443 -- 455.

\bibitem{Arch}
R. J. Archbold, \emph{On the `flip-flop automorphism of $C^*(S1,S2)$}, 
Quart. J. Math. Oxford Ser. (2) \textbf{30} (1979), no. 118, 129-–132.

\bibitem{BM} F. P. Boca, C. Merriman, \emph{Coding of geodesics on some modular surfaces and applications to odd and even
continued fractions}, Indag. Math. (N.S.) \textbf{29} (2018), no. 5, 1214–-1234.
\bibitem{BL} F. P. Boca, C. Linden, \emph{On Minkowski type question mark functions associated with even or odd
continued fractions}, Monatsh. Math. \textbf{187} (2018), no. 1, 35-–57.


\bibitem{BI}
C. Bonnano, S. Isola, \emph{Orderings of the rationals and dynamical systems},
Colloq. Math. \textbf{116} (2009), 165--189.

\bibitem{BraJorg}
O. Bratteli,  P.E.T. Jorgensen, \emph{Iterated function systems and permutation representations of the Cuntz algebra}, Mem. Amer. Math. Soc. \textbf{139} (1999), no. 663
\bibitem{BY}
 G. Brown, Q. Yin. \emph{Metrical theory for Farey continued fractions}, Osaka J. Math. 33(4) (1996), 951--970
\bibitem{ContiSyz}
R. Conti, W. Szymański, \emph{Labeled trees and localized automorphisms of the Cuntz algebras}, Trans. Amer. Math. Soc. \textbf{363} (2011), no. 11, 5847–-5870. 

\bibitem{Cu}
J. Cuntz, \emph{Simple $\text{C}^*$-algebras generated by isometries}. 
Comm. Math. Phys. \textbf{57} (1977), no. 2, 173–-185. 

\bibitem{H}
B. Heersink, \emph{An effective estimate for the Lebesgue measure of preimages of iterates of the Farey map}, Adv. Math. \textbf{291} (2016), 621--634.

\bibitem{Hua}
L. K. Hua, \emph{Introduction to Number Theory}, Springer, New York, 1982.


\bibitem{Kawa4}
K. Kawamura, \emph{Polynomial embedding of Cuntz-Krieger algebra into
Cuntz algebra}. preprint RIMS-1391 (2003).


\bibitem{Kawa3}
K. Kawamura,  \emph{Polynomial endomorphisms of the Cuntz algebras arising from permutations}, I. General theory. Lett. Math. Phys. \textbf{71} (2005), no. 2, 149–-158.
\bibitem{KawaHayaLascu}
K. Kawamura,  Y. Hayashi, D. Lascu, \emph{Continued fraction expansions and permutative representations of the Cuntz algebra $\OO_{\infty}$}, J. Number Theory \textbf{129} (2009), no. 12, 3069–-3080

\bibitem{KawaLascuColt}
K. Kawamura, D. Lascu, I. Coltescu, \emph{Jump transformations and an embedding of $\OO_{\infty}$ into $\OO_2$}. J. Math. Phys. \textbf{50} (2009), no. 3, 033501
\bibitem{Kra}
C. Kraaikamp, \emph{A new class of continued fraction expansions},
Acta Arith. \textbf{57} (1991), 1--39.

\bibitem{KMS}
M. Kesseb\" ohmer, S. Munday, B. O. Stratmann, \emph{Infinite Ergodic Theory of Numbers},
De Gruyter Graduate, De Gruyter, Berlin, 2016.


\bibitem{Panti2}
G. Panti, \emph{A General Lagrange Theorem}, Amer. Math. Monthly \textbf{116}, (2009), no. 1, 70--74.  


\bibitem{Panti}
G. Panti, \emph{Slow continued fractions, transducers, and the Serret theorem}, J. Number Theory \textbf{185} (2018), 121-–143.






\bibitem{Sch1}
F. Schweiger, \emph{Continued fractions with odd and even partial quotients},
Arbeitsberichte Math. Institut Universit\" at Salzburg \textbf{4} (1982), 59--70.

\bibitem{Serret}
J.A. Serret, \emph{Cours d'alg\`ebre sup\'erieure}, 3rd edition, Gauthier-Villars, 1866.

\bibitem{Williams}
D.P. Williams, \emph{Crossed Products of $C^*$-Algebras}, AMS, Providence, USA, 2007.


\bibitem{Zag}
D. Zagier, \emph{Nombres de classes et fractions continues}, Astérisque \textbf{24-25} (1975) 81-–97.
\end{thebibliography}
\end{document}